\documentclass[12pt]{article}
\usepackage{graphicx,amsthm,amsmath,amsfonts,amssymb}

\newtheorem{theorem}{Theorem}

\newcommand{\R}{\mathbb{R}}
\newcommand{\bd}[1]{\boldsymbol{#1}}
\newcommand{\V}{\mathrm{V}} 
\newcommand{\M}{\mathrm{M}}
\newcommand{\dx}{\mathrm{d}x}

\begin{document}

\title{Polyconvexity and Existence Theorem for Nonlinearly Elastic Shells}
\author{Sylvia Anicic}
\date{}


\maketitle
\begin{center}
\small{
Universit\'e de Haute-Alsace, France\\
Email: sylvia.anicic@uha.fr}
\end{center}

\begin{abstract}
We present an existence theorem for a large class of nonlinearly
elastic shells with low regularity in the framework of a two-dimensional theory involving the mean and Gaussian curvatures. We restrict our discussion to hyperelastic materials, that is to elastic materials possessing a stored energy function. Under some specific conditions of polyconvexity, coerciveness and growth of the stored energy function, we prove the existence of global minimizers. In addition, we define a general class of polyconvex stored energy functions which satisfies a coerciveness inequality.
\end{abstract}


49J20, 35A01, 35Q74

\section{Introduction}
\label{Introduction}

A shell is a three-dimensional elastic body which occupies a volume contained between two surfaces (in general parallel) close to each other. A natural way to define a shell is to consider a surface $S$ embedded in $\R^3$ and to thicken it on each side.
In response to given loads, the displacement and the stress arising in an elastic shell, viewed as a three-dimensional body, are predicted by the equations of nonlinear three-dimensional elasticity. To this day, there are two theories of existence of solutions for these equations: one based on the implicit function theorem and the other, due to a seminal paper of Ball \cite{Ball77}, based on the minimization of functionals. This latter asserts that if the constituting material is hyperelastic and the associated stored energy function satisfies some specific conditions of convexity (called polyconvexity), coerciveness and growth, the minimization problem has at least one solution.

As a shell is "almost" a surface and may even be ultrathin such as polymer films or biological membranes, shell modeling is part of a two-dimensional theory involving only the deformation of the surface $S$. This approach yields a variety of two-dimensional nonlinear shell models, which can be classified into two categories. 

A first category consists of two-dimensional nonlinear shell equations obtained from the three-dimen\-sional elasticity by means of an asymptotic analysis when the thickness goes to zero. The question of how to rigorously identify and justify the nonlinear two-dimensional shell equations from the three-dimensional elasticity was finally settled in two key contributions, one by Le Dret \& Raoult \cite{LeDretRaoult95}
 and  one by Friesecke, James, Mora \& M\"uller \cite{FrieseckeJames03}, who respectively justified the equations of a nonlinearly elastic membrane shell and those of a nonlinearly elastic flexural shell through the use of $\Gamma$-convergence theory. This theory automatically provides the existence of a minimizer for the $\Gamma$-limit functional. Specifically for the nonlinearly elastic flexural shell equations, Ciarlet \& Coutand \cite{CiarletCoutand98} have established the existence of a minimizer by direct methods in calculus of variations.

A second category consists of two-dimensional nonlinear shell models obtained from the three-dimensional elasticity by restricting the range of admissible deformations and stresses by means of specific a priori assumptions such as Cosserat assumptions (Simo \& Fox \cite{SimoFox89}) or Kirchhoff-Love assumptions (Koiter \cite{Koiter66}). The topic of existence of solutions for these models has been treated for various types of shells and with different techniques in the literature (Antman \cite{Antman76-1,Antman76-2}, Ciarlet \& Gratie \cite{CiarletGratie06}, Ciarlet, Gogu \& Mardare \cite{CiarletGogu13}, B\^{\i}rsan \& Neff \cite{BirsanNeff14}, Bunoiu, Ciarlet \& Mardare \cite{BunoiuCiarlet15} and Ciarlet \& Mardare \cite{CiarletMardare16}).

In Sect.~\ref{Existence}, we present a general theorem of existence of global minimizers for nonlinear shells
in the framework of a two-dimensional theory involving the mean and Gaussian curvatures. Inspired by the approach of Ball \cite{Ball77}, we define a notion of a polyconvex and orientation-preserving 
stored energy function for shells. As an example, the Helfrich \cite{Helfrich73} density energy function used by the mechanical community for modeling biological membranes, is polyconvex but not orientation-preserving. In Sect.~\ref{Example}, we introduce a class of polyconvex stored energy functions for $G^1$ shells which satisfies a coerciveness inequality.   

\section{Notations}
\label{Notations}

In all that follows, Greek indices and exponents range in the set $\{1,2\}$ while Latin indices and exponents range in the set $\{1,2,3\}$ (except when they are used for indexing sequences). We use the Einstein summation convention with respect to repeated indices and exponents.

The three-dimensional Euclidean space is identified with $\R^3$ by choosing an origin and a Euclidean basis. Vector and tensor fields are denoted by boldface letters. The Euclidean norm, the inner product, the vector product and the tensor product of two vectors $\bd{u}$ and $\bd{v}$ in 
$\R^3$ are respectively denoted $|\bd{u}|$, $\bd{u}\cdot \bd{v}$, 
$\bd{u}\wedge \bd{v}$ and $\bd{u} \otimes \bd{v}$. The sets of all $m\times n$ real matrices are denoted $\mathbb{M}^{m\times n}$. For a real matrix $\bd{A}\in \mathbb{M}^{m\times n}$, the notation $|\bd{A}|:=\mathrm{tr}\,(\bd{A}^{\mathrm{T}}\bd{A})^{1/2}$ stands for the Frobenius norm.

A domain $\omega \subset \R^2$ is a bounded, connected, open set with a Lipschitz-continuous boundary $\gamma :=\partial \omega$, the set $\omega$ being locally on the same side of $\gamma$. A generic point in the set 
$\overline{\omega}$ is denoted by $x=(x_{\alpha})$ and partial derivatives, in the classical or distributional sense, are denoted $\partial_{\alpha}:=\partial / \partial x_{\alpha}$.

The notation $L^p(\omega;\R^3)$ with $1\leqslant p <\infty$ designates the space of vector fields $\bd{\xi}=(\xi_i):\omega \rightarrow \R^3$ with components $\xi_i$ in the usual Lebesgue space $L^p(\omega)$. It is equipped with the norm 
\[
\|\bd{\xi}\|_{p}:=\left(\int_{\omega}|\bd{\xi}(x)|^p\, \dx \right)^{1/p}\quad
\textnormal{for any }\bd{\xi}\in L^p(\omega;\R^3).
\]
The space $W^{1,p}(\omega;\R^3)$ denotes the space of vector fields
$\bd{\xi}=(\xi_i):\omega \rightarrow \R^3$ with components $\xi_i$ in the usual Sobolev space $W^{1,p}(\omega)$. It is equipped with the norm
\[
\|\bd{\xi}\|_{1,p}:=\left(\|\bd{\xi}\|_{p}^p+\sum_{\alpha=1}^2\|\partial_{\alpha}\bd{\xi}\|^p_{p} \right)^{1/p}\quad
\textnormal{for any }\bd{\xi}\in W^{1,p}(\omega;\R^3).
\]
The space $W^{1,\infty}(\omega;\R^3)$ consists of vector fields
$\bd{\xi}=(\xi_i):\omega \rightarrow \R^3$ with components $\xi_i$ in the usual Sobolev space $W^{1,\infty}(\omega)$ of Lipschitz continuous functions on $\overline{\omega}$.

Strong and weak convergences are respectively denoted $\rightarrow$ and
$\rightharpoonup$.

\section{An existence theorem}
\label{Existence}

First, let us briefly recall the framework considered in the context of three-dimen\-sional elasticity. Let $\Omega\subset \R^3$ be a domain considered as the reference configuration of an elastic body. The admissible deformations $\bd{\Theta}:\Omega \rightarrow \R^3$
satisfy
\[
\det \nabla \bd{\Theta}>0.
\]
Now we consider a shell $\mathcal{C}$ with thickness $2\varepsilon>0$ whose reference configuration is the set
\[
\mathcal{C}=\{\bd{\Phi}(x,z)=\bd{\varphi}(x)+z\bd{a}_3(x),\quad (x,z)\in \Omega:=
\omega \times (-\varepsilon,\varepsilon)\}
\]
where $\omega\subset \R^2$ is a domain and 
\[
\bd{a}_3(x):=\frac{\partial_1 \bd{\varphi}(x)\wedge
\partial_2 \bd{\varphi}(x) }{|\partial_1 \bd{\varphi}(x)\wedge
\partial_2 \bd{\varphi}(x)|}
\]
is the unit normal vector to the midsurface $S:=\bd{\varphi}(\omega)$.
We make the realistic assumption that the deformations $\bd{\Theta}:\mathcal{C}\rightarrow \R^3$
of the shell are of the form
\[
\bd{\Theta}(\bd{\Phi}(x,z))=\bd{\psi}(x)+z\bd{a}_3(\bd{\psi})(x),\quad (x,z)\in \Omega,
\]
where 
\[
\bd{a}_3(\bd{\psi})(x):=\frac{\partial_1 \bd{\psi}(x)\wedge
\partial_2 \bd{\psi}(x) }{|\partial_1 \bd{\psi}(x)\wedge
\partial_2 \bd{\psi}(x)|}
\]
is the unit normal vector to the deformed midsurface $\hat{S}:=
\bd{\psi}(\omega)$. By letting
\[
\bd{\Psi}(x,z):=\bd{\Theta}\circ \bd{\Phi}(x,z)=\bd{\psi}(x)+z\bd{a}_3(\bd{\psi})(x),
\]
it follows that
\[
\det \nabla \bd{\Psi}(x,z)=\det\nabla \bd{\Theta}(\bd{\Phi}(x,z))\det \nabla \bd{\Phi}(x,z).
\]
Hence, in order to satisfy the condition $\det\nabla \bd{\Theta}(\bd{\Phi}(x,z))>0$, we require that
\[
\det \nabla \bd{\Phi}>0\quad \text{and}\quad \det \nabla \bd{\Psi}>0.
\]
Thus, since
\begin{align*}
\det \nabla \bd{\Psi}=
\left(1-\frac{z}{{R}_1(\bd{\psi})}\right)\left(1-\frac{z}{{R}_2(\bd{\psi})}\right)|\partial_1 \bd{\psi}\wedge
\partial_2 \bd{\psi}|
\end{align*}
where $1/{R}_1(\bd{\psi})$ and $1/{R}_2(\bd{\psi})$ are the principal curvatures of the deformed midsurface, we impose the following conditions
\[
\partial_1 \bd{\psi}\wedge
\partial_2 \bd{\psi}\neq \bd{0}\quad \text{and}\quad
\max_{\alpha \in \{1,2\}}\left|\frac{\varepsilon}{{R}_{\alpha}(\bd{\psi})}\right|<1.
\]
We denote by
\[
a(\bd{\psi}):=|\partial_1 \bd{\psi}\wedge \partial_2 \bd{\psi}|^2=\det(a_{\alpha \beta}(\bd{\psi})),
\]
where $a_{\alpha \beta}(\bd{\psi}):=\partial_{\alpha}\bd{\psi}\cdot
\partial_{\beta}\bd{\psi}$, and if $a(\bd{\psi})\neq 0$, we denote by
\[
H(\bd{\psi}):=\frac{1}{2}\left(\frac{1}{R_1(\bd{\psi})}+\frac{1}{R_2(\bd{\psi})}\right)\quad \text{and}\quad K(\bd{\psi}):=\frac{1}{R_1(\bd{\psi})R_2(\bd{\psi})}
\]
the mean and Gaussian curvatures. The principal curvatures
$1/R_1(\bd{\psi})$ and $1/R_2(\bd{\psi})$ are the two eigenvalues of the matrix
$(b_{\alpha}^{\beta}(\bd{\psi}))$ defined as
$b_{\alpha}^{\beta}(\bd{\psi}):=b_{\alpha \rho}(\bd{\psi})a^{\rho \beta}(\bd{\psi})$ with $b_{\alpha \beta}(\bd{\psi}):=-\partial_{\alpha}\bd{\psi}\cdot \partial_{\beta}\bd{a}_3(\bd{\psi})$ and
$(a^{\alpha \beta}(\bd{\psi})):=(a_{\alpha \beta}(\bd{\psi}))^{-1}$.
 
\begin{theorem}
Let $\omega$ be a domain in $\R^2$ and let 
$\gamma_0$ be a non-empty relatively open subset of $\gamma:=\partial \omega$.
For $\varepsilon>0$, $p\geqslant 2$ and $q>1$, we define the functional $I:\bd{\V}^{\varepsilon}\rightarrow \R\cup \{+ \infty\}$ by letting
\begin{align*}
\bd{\V}^{\varepsilon}:=\{ & \bd{\psi}\in W^{1,p}(\omega;\R^3); \quad
\sqrt{a( \bd{\psi})} \in L^q(\omega),\quad
a( \bd{\psi})\neq 0 \textnormal{ a.e. in }\omega,\\
&\bd{a}_3 (\bd{\psi}) \in W^{1,p}(\omega;\R^3), \quad \max\left\{\left|\varepsilon/R_{1}(\bd{\psi})\right|,\left|\varepsilon/R_{2}(\bd{\psi})\right|\right\}<1 \textnormal{ a.e. in }\omega,\\
 & \bd{\psi}= \bd{\varphi} \quad \text{and} \quad \bd{a}_3 (\bd{\psi})  =\bd{a}_3 \text{ }\mathrm{d}\gamma\text{-a.e. in } \gamma_0  \}
\end{align*}
and for each $\bd{\psi}\in \bd{\V}^{\varepsilon}$,
\[
I(\bd{\psi}):=\int_{\omega}W(x,\bd{\psi})\, \dx-L(\bd{\psi},\bd{a}_3 (\bd{\psi})),
\]
where $L$ is a continuous linear form over the space $W^{1,p}(\omega;\R^3)\times W^{1,p}(\omega;\R^3)$ and 
$W:\omega \times \bd{\V}^{\varepsilon} \rightarrow \R$ is a function with the following properties:

\textnormal{(a)} Polyconvexity: For almost all $x \in \omega$, there exists a convex function $\mathbb{W}(x,\cdot):\bd{\M}\rightarrow \R$ where
\[
\bd{\M}:=\{( \bd{A},
\bd{B},a,b,c) \in (\mathbb{M}^{3\times 2})^2 \times \R^3; \,
 a-|b|>0 \text{ and } a-2|b| +c >0 \}
\]
such that for almost all $x \in \omega$
\[
 W(x,\bd{\psi})=
\mathbb{W}\Big(x,\nabla \bd{\psi}(x),
\nabla \bd{a}_{3}(\bd{\psi})(x),
\big(1,\varepsilon H(\bd{\psi}(x)),\varepsilon^2 K(\bd{\psi}(x))\big)\sqrt{a(\bd{\psi}(x))}\Big).
\]

\textnormal{(b)} Measurability: The function $\mathbb{W}(\cdot, \bd{A},
\bd{B},a,b,c):\omega \rightarrow \R$ is measurable for all $(\bd{A},
\bd{B},a,b,c)\in
\bd{\M}$.

\textnormal{(c)} Coerciveness: There exist constants $C_1>0$ and $C_2$ such that
\[
W(x,\bd{\psi})\geqslant C_1\{|\nabla \bd{\psi}|^p+|\nabla \bd{a}_{3}(\bd{\psi})|^p+a( \bd{\psi})^{q/2}\}+C_2
\]
for all $\bd{\psi} \in \bd{\V}^{\varepsilon}$ and almost all $x\in \omega$.

\textnormal{(d)} {Orientation-preserving condition}: 
\begin{align*}
&W(x,\bd{\psi})\rightarrow +\infty \, \text{as } \{1-2\varepsilon H(\bd{\psi}(x))+\varepsilon^2 K(\bd{\psi}(x))\}\sqrt{a(\bd{\psi}(x))}\rightarrow   0^+\\
\text{and }&W(x,\bd{\psi})\rightarrow +\infty \, \text{as } \{1+2\varepsilon H(\bd{\psi}(x))+\varepsilon^2 K(\bd{\psi}(x))\}\sqrt{a(\bd{\psi}(x))}\rightarrow   0^+
\end{align*}
for all $\bd{\psi} \in \bd{\V}^{\varepsilon}$ and almost all $x\in \omega$.

Assume that $\inf_{\bd{\psi}\in \bd{\V}^{\varepsilon}}I(\bd{\psi})<+\infty$, then there exists at least one function $\bd{\eta}\in \bd{\V}^{\varepsilon}$ such that
\[
I(\bd{\eta})=\inf_{\bd{\psi}\in \bd{\V}^{\varepsilon}}I(\bd{\psi}).
\]
\end{theorem}

\begin{proof}
(i) \emph{The integrals $\int_{\omega}W(x,\bd{\psi})\, \dx$ are well defined for all $\bd{\psi}\in \bd{\V}^{\varepsilon}$. } First, we note that the set $\bd{\M}$ is a convex open subset of $(\mathbb{M}^{3\times 2})^2 \times \R^3$. Furthermore, each $\bd{\psi} \in {\bd{\V}^{\varepsilon}}$ satisfies
$a(\bd{\psi}(x))>0$ and
$| \varepsilon / R_{\alpha}(\bd{\psi}(x))|<1$, then  
for almost all $x\in \omega$, 
\[
\big(\nabla \bd{\psi}(x),
\nabla \bd{a}_{3}(\bd{\psi})(x),
\big(1,\varepsilon H(\bd{\psi}(x)),\varepsilon^2 K(\bd{\psi}(x))\big)\sqrt{a(\bd{\psi}(x))}\big)\in \bd{\mathrm{M}}.
\]
In addition, for almost all $x\in \omega$, the function $\mathbb{W}(x,\cdot):
\bd{\M}\rightarrow \R$ is continuous and for all
$(\bd{A},
\bd{B},a,b,c)\in
\bd{\M}$, the function $\mathbb{W}(\cdot, \bd{A},
\bd{B},a,b,c):\omega \rightarrow \R$ is measurable. Hence,
$\mathbb{W}:\omega \times \bd{\M} \rightarrow \R$ is a Carath\'eodory function, and thus the function 
\[
x\in \omega \rightarrow \mathbb{W}(x,\nabla \bd{\psi}(x),
\nabla \bd{a}_{3}(\bd{\psi})(x),
\alpha(x),\beta(x),\gamma(x))\in \R
\]
with
$\alpha(x):=\sqrt{a(\bd{\psi}(x))}$, $\beta(x):=\varepsilon H(\bd{\psi}(x))\alpha(x)$ and 
$\gamma(x):=\varepsilon^2K(\bd{\psi}(x))\alpha(x)$
is measurable for each $\bd{\psi} \in {\bd{\V}^{\varepsilon}}$. The function 
$W$ being in addition bounded from below (by the coerciveness inequality (c)), the integral
\[
\int_{\omega}W(x,\bd{\psi})\, \dx=\int_{\omega}
\mathbb{W}(x,\nabla \bd{\psi}(x),
\nabla \bd{a}_{3}(\bd{\psi})(x),
\alpha(x),\beta(x),\gamma(x))\, \dx
\]
is therefore a well-defined extended real number in the interval
$[C_2 \text{ area }\omega,+\infty]$ for each $\bd{\psi} \in {\bd{\V}^{\varepsilon}}$.
\medskip

(ii) \emph{We find a lower bound for $I(\bd{\psi})$ when 
$\bd{\psi} \in {\bd{\V}^{\varepsilon}}$}.

From the assumed coerciveness (c) of the function $W$ and the assumed continuity of the linear form $L$, we infer that there exists a constant $C_3>0$ such that
\begin{align*}
I(\bd{\psi})\geqslant C_1\int_{\omega}
\{|\nabla \bd{\psi}|^p&+|\nabla \bd{a}_{3}(\bd{\psi})|^p+a(\bd{\psi})^{q/2}\}\, \dx+C_2 \text{ area }\omega \\
&-C_3(\| \bd{\psi}\|_{1,p}+\| \bd{a}_3(\bd{\psi})\|_{1,p})\text{ for all }\bd{\psi} \in {\bd{\V}^{\varepsilon}}.
\end{align*}
Combining the boundary conditions $
\bd{\psi}= \bd{\varphi}$ and $\bd{a}_3 (\bd{\psi})  =\bd{a}_3$
on $\gamma_0$ with the generalized Poincar\'e inequality, we thus conclude that there exist constants $C_4>0$ and $C_5$ such that
\[
I(\bd{\psi})\geqslant C_4\{
\| \bd{\psi}\|^p_{1,p}+\|\bd{a}_{3}(\bd{\psi})\|^p_{1,p}+\|\sqrt{a(\bd{\psi})}\|^{q}_{q}\}+C_5\text{ for all }\bd{\psi} \in {\bd{\V}^{\varepsilon}}.
\]

(iii) \emph{
We show that if $(\bd{\eta}^k)$ is a sequence with $\bd{\eta}^k
\in {\bd{\V}^{\varepsilon}}$ for all $k$ for which there exist
$\bd{\eta}\in W^{1,p}(\omega;\R^3)$, $\bd{\kappa}\in W^{1,p}(\omega;\R^3)$,  $(\bd{\xi}_1,\bd{\xi}_2,\bd{\xi}_3)\in (L^{q}(\omega;\R^3))^3$ and $(\alpha_1,\alpha_2,\alpha_3)\in (L^{q}(\omega))^3$ such that }
\begin{align*}
&\bd{\eta}^k \rightharpoonup \bd{\eta} \text{ in }W^{1,p}(\omega;\R^3),
\quad \bd{a}_3(\bd{\eta}^k) \rightharpoonup \bd{\kappa} \text{ in }W^{1,p}(\omega;\R^3),\\
&\partial_1\bd{\eta}^k \wedge \partial_2\bd{\eta}^k \rightharpoonup \bd{\xi}_1 \text{ in }L^{q}(\omega;\R^3), \quad \sqrt{a(\bd{\eta}^k)} \rightharpoonup \alpha_1 \text{ in }L^{q}(\omega), \\
&H(\bd{\eta}^k)\partial_1\bd{\eta}^k \wedge \partial_2\bd{\eta}^k \rightharpoonup \bd{\xi}_2 \text{ in }L^{q}(\omega;\R^3), \quad
H(\bd{\eta}^k)\sqrt{a(\bd{\eta}^k)} \rightharpoonup \alpha_2 \text{ in }L^{q}(\omega),\\
&K(\bd{\eta}^k)\partial_1\bd{\eta}^k \wedge \partial_2\bd{\eta}^k \rightharpoonup \bd{\xi}_3 \text{ in }L^{q}(\omega;\R^3), \quad
K(\bd{\eta}^k)\sqrt{a(\bd{\eta}^k)} \rightharpoonup \alpha_3 \text{ in }L^{q}(\omega),
\end{align*}
\emph{then almost everywhere in $\omega$}
\begin{align*}
&\bd{\kappa}=\bd{a}_3 (\bd{\eta}),\quad \max\{\left|\varepsilon /R_{1}(\bd{\eta})\right|,\left|\varepsilon /R_{2}(\bd{\eta})\right|\}\leqslant 1, \\ 
&\bd{\xi}_1=\partial_1\bd{\eta}\wedge \partial_2\bd{\eta},\quad
\bd{\xi}_2=H(\bd{\eta})\partial_1\bd{\eta}\wedge \partial_2\bd{\eta},\quad
\bd{\xi}_3=K(\bd{\eta})\partial_1\bd{\eta}\wedge \partial_2\bd{\eta},\\
& \alpha_1=\sqrt{a(\bd{\eta})},\quad \alpha_2=H(\bd{\eta})\sqrt{a(\bd{\eta})} \quad \text{and} \quad
\alpha_3=K(\bd{\eta})\sqrt{a(\bd{\eta})}.
\end{align*}
To prove this assertion, we begin by showing that $\bd{\kappa}=\bd{a}_3 (\bd{\eta})$. Using the Rellich-Kondra$\check{\text{s}}$ov compact imbedding theorem
$W^{1,p}(\omega;\R^3)\Subset L^r(\omega;\R^3)$ for all $r$ with $1\leqslant r < \infty$, we infer that
\[
\bd{a}_3(\bd{\eta}^k)\rightarrow \bd{\kappa}  
\text{ in }L^{p'}(\omega;\R^3),\quad \frac{1}{p}+\frac{1}{p'}=1,
\text{ and } \bd{a}_3(\bd{\eta}^k)\rightarrow \bd{\kappa}  
\text{ in }L^{2}(\omega;\R^3).
\]
Hence
$
\partial_{\alpha}\bd{\eta}^k\cdot \bd{a}_3(\bd{\eta}^k)\rightharpoonup
\partial_{\alpha}\bd{\eta}\cdot \bd{\kappa}$ in $L^{1}(\omega)$ and 
$|\bd{a}_3(\bd{\eta}^k)|^2\rightarrow |\bd{\kappa}|^2$ in $L^{1}(\omega)$.
Since for all $k$, $\partial_{\alpha}\bd{\eta}^k\cdot \bd{a}_3(\bd{\eta}^k)=0$ and $|\bd{a}_3(\bd{\eta}^k)|=1$, it follows that
$
\partial_{\alpha}\bd{\eta}\cdot \bd{\kappa}=0$ and  $
|\bd{\kappa}|=1$. In order to prove that $\bd{\kappa}=\bd{a}_3(\bd{\eta})$, it remains to show that
\[
\partial_1 \bd{\eta}\wedge \partial_2 \bd{\eta}\cdot \bd{\kappa}\geqslant 0 \text{ a.e. on }\omega.
\]
To this end, we define 
for all $\bd{\varphi}_1 \in W^{1,p}(\omega;\R^3)$ and all
$\bd{\varphi}_2 \in W^{1,p}(\omega;\R^3)$
\begin{align*}
\left[ \bd{\varphi}_1,\bd{\varphi}_2\right]:=&\frac{1}{2}\left(\partial_1 \bd{\varphi}_1\wedge \partial_2 \bd{\varphi}_2+
 \partial_1 \bd{\varphi}_2 \wedge \partial_2 \bd{\varphi}_1\right) \\
 =& \frac{1}{4}\left\{
\partial_1(\bd{\varphi}_1\wedge \partial_2\bd{\varphi}_2+
\bd{\varphi}_2\wedge \partial_2 \bd{\varphi}_1)+\partial_2(\partial_1 \bd{\varphi}_1\wedge \bd{\varphi}_2+\partial_1 \bd{\varphi}_2\wedge
\bd{\varphi}_1) 
\right\}.
\end{align*}
Hence, if $(\bd{\varphi}_1^k,\bd{\varphi}_2^k)$ is a sequence of
$W^{1,p}(\omega;\R^3)$, $p\geqslant 2$,
which converges weakly to $(\bd{l}_1,\bd{l}_2)\in 
W^{1,p}(\omega;\R^3)$, then
$\left[ \bd{\varphi}_1^k,\bd{\varphi}_2^k\right] \rightharpoonup
\left[ \bd{l}_1,\bd{l}_2\right]$ in $\mathcal{D}'(\omega;\R^3)$.
By applying this result to the sequence $(\bd{\eta}^k)$ which converges weakly to $\bd{\eta}$ in
$W^{1,p}(\omega;\R^3)$, it follows that
\begin{align*}
[\bd{\eta}^k,\bd{\eta}^k]=\partial_1 \bd{\eta}^k\wedge \partial_2 \bd{\eta}^k \rightharpoonup [\bd{\eta},\bd{\eta}]=\partial_1 \bd{\eta}\wedge \partial_2 \bd{\eta}\text{ in }\mathcal{D}'(\omega;\R^3)
\end{align*}
Hence
$\bd{\xi}_1=\partial_1 \bd{\eta}\wedge \partial_2 \bd{\eta}$ and $\partial_1 \bd{\eta}^k\wedge \partial_2 \bd{\eta}^k
\rightharpoonup \partial_1 \bd{\eta}\wedge \partial_2 \bd{\eta}$ in $L^q(\omega;\R^3)$.
Then
\[
\sqrt{a(\bd{\eta}^k)}=\partial_1 \bd{\eta}^k\wedge \partial_2 \bd{\eta}^k\cdot 
\bd{a}_3(\bd{\eta}^k)\rightharpoonup 
\partial_1 \bd{\eta}\wedge \partial_2 \bd{\eta}\cdot 
\bd{\kappa} \text{ in } L^1(\omega).
\]
Since for all $k$, $\partial_1 \bd{\eta}^k\wedge \partial_2 \bd{\eta}^k\cdot 
\bd{a}_3(\bd{\eta}^k)> 0$
then
$
\partial_1 \bd{\eta}\wedge \partial_2 \bd{\eta}\cdot 
\bd{\kappa}\geqslant 0$ a.e. in $\omega$.
Combining the following three relations, 
\[
\partial_\alpha \bd{\eta}\cdot \bd{\kappa}=0,\quad |\bd{\kappa}|=1, \quad \text{and}\quad 
\partial_1 \bd{\eta}\wedge \partial_2 \bd{\eta}\cdot 
\bd{\kappa}\geqslant 0 \text{ a.e. in }\omega,
\]
we infer that
\[
\bd{\kappa}=\frac{\partial_1 \bd{\eta} \wedge \partial_2 \bd{\eta}}{|\partial_1 \bd{\eta} \wedge \partial_2 \bd{\eta}|}=\bd{a}_3(\bd{\eta})\quad \text{and}\quad \alpha_1=\partial_1 \bd{\eta}\wedge \partial_2 \bd{\eta}\cdot 
\bd{\kappa}=\sqrt{a(\bd{\eta})}.
\]
Similarly, since 
$(\bd{\eta}^k,\bd{a}_3(\bd{\eta}^k))\rightharpoonup(\bd{\eta},\bd{a}_3(\bd{\eta}))$ in $W^{1,p}(\omega;\R^3)$,
it follows that
\begin{align*}
&[\bd{\eta}^k,\bd{a}_3(\bd{\eta}^k)]=-H(\bd{\eta}^k)\partial_1 \bd{\eta}^k\wedge \partial_2 \bd{\eta}^k \rightharpoonup [\bd{\eta},\bd{a}_3(\bd{\eta})]=-H(\bd{\eta})\partial_1 \bd{\eta}\wedge \partial_2 \bd{\eta},\\
&[\bd{a}_3(\bd{\eta}^k),\bd{a}_3(\bd{\eta}^k)] =K(\bd{\eta}^k)\partial_1 \bd{\eta}^k\wedge \partial_2 \bd{\eta}^k\rightharpoonup [\bd{a}_3(\bd{\eta}),\bd{a}_3(\bd{\eta})]=K(\bd{\eta})\partial_1 \bd{\eta}\wedge \partial_2 \bd{\eta}
\end{align*}
in $\mathcal{D}'(\omega;\R^3)$. Hence $\bd{\xi}_2=H(\bd{\eta})\partial_1 \bd{\eta}\wedge \partial_2 \bd{\eta}$, $\bd{\xi}_3=K(\bd{\eta})\partial_1 \bd{\eta}\wedge \partial_2 \bd{\eta}$ and
\begin{align*}
&H(\bd{\eta}^k)\partial_1 \bd{\eta}^k\wedge \partial_2 \bd{\eta}^k
\cdot \bd{a}_3(\bd{\eta}^k)
\rightharpoonup 
H(\bd{\eta})\partial_1 \bd{\eta}\wedge \partial_2 \bd{\eta}\cdot 
\bd{a}_3(\bd{\eta})=H(\bd{\eta})\sqrt{a(\bd{\eta})},\\
&K(\bd{\eta}^k)\partial_1 \bd{\eta}^k\wedge \partial_2 \bd{\eta}^k
\cdot \bd{a}_3(\bd{\eta}^k)
\rightharpoonup 
K(\bd{\eta})\partial_1 \bd{\eta}\wedge \partial_2 \bd{\eta}\cdot 
\bd{a}_3(\bd{\eta})=K(\bd{\eta})\sqrt{a(\bd{\eta})}
\end{align*}
in $L^1(\omega)$. Then
$\alpha_2=H(\bd{\eta})\sqrt{a(\bd{\eta})}$ and $\alpha_3=
K(\bd{\eta})\sqrt{a(\bd{\eta})}$.

It remains to show that for all $\alpha\in\{1,2\}$,
$
|\varepsilon/R_{\alpha}(\bd{\eta})|\leqslant 1
\text{ a.e. in }\omega.
$
Combining all the previous relations leads to the following weak convergence in $L^q(\omega)$, for all $d\in\{-1,1\}$,
\begin{align*}
(1-d\varepsilon H(\bd{\eta}^k))\sqrt{a(\bd{\eta}^k)}
&\rightharpoonup (1-d\varepsilon H(\bd{\eta}))\sqrt{a(\bd{\eta})}, \\
(1-2d\varepsilon H(\bd{\eta}^k)+\varepsilon ^2 K(\bd{\eta}^k))\sqrt{a(\bd{\eta}^k)}&\rightharpoonup
(1-2d\varepsilon H(\bd{\eta})+\varepsilon ^2 K(\bd{\eta}))\sqrt{a(\bd{\eta})}.
\end{align*}
Since for all $k$ and all $\alpha\in\{1,2\}$,
$\sqrt{a(\bd{\eta}^k)}>0$ and $|\varepsilon/R_{\alpha}(\bd{\eta}^k)|< 1$ a.e. in $\omega$,
then for all $k$ and all $d\in\{-1,1\}$,
$
(1-d\varepsilon H(\bd{\eta}^k))\sqrt{a(\bd{\eta}^k)}>0 $ and 
$
(1-2d\varepsilon H(\bd{\eta}^k)+\varepsilon ^2 K(\bd{\eta}^k))\sqrt{a(\bd{\eta}^k)}>0$  a.e. in $\omega$,
then by passing to the weak limit in $L^q(\omega)$, it follows that for all
$d\in\{-1,1\}$,
$(1-d\varepsilon H(\bd{\eta}))\sqrt{a(\bd{\eta})}\geqslant 0$ and $(1-2d\varepsilon H(\bd{\eta})+\varepsilon ^2 K(\bd{\eta}))\sqrt{a(\bd{\eta})}\geqslant 0$ a.e. in $\omega$. 
Hence for all $\alpha\in \{1,2\}$
$
|\varepsilon/R_{\alpha}(\bd{\eta})|\leqslant 1$
a.e. in $\omega$.

(iv) \emph{Let $(\bd{\eta}^k)$ be an infimizing sequence for the functional $I$, i.e., a sequence that satisfies}
\[
\bd{\eta}^k\in {\bd{\V}^{\varepsilon}} \text{ for all }k, \quad \text{and} \quad \lim_{k\rightarrow \infty}I(\bd{\eta}^k)=\inf_{\bd{\psi}\in {\bd{\V}^{\varepsilon}}}I(\bd{\psi}).
\]
By assumption, $\inf_{\bd{\psi}\in {\bd{\V}^{\varepsilon}}}I(\bd{\psi})<+\infty$, and thus, by part (ii), the sequence $(\bd{\eta}^k,\bd{a}_3(\bd{\eta}^k))$ is bounded in $(W^{1,p}(\omega;\R^3))^2$ and the sequence $\sqrt{a(\bd{\eta}^k)}$ is  bounded in $L^{q}(\omega)$. 
Since
\[
\sqrt{a(\bd{\eta}^k)}=\partial_1\bd{\eta}^k\wedge \partial_2\bd{\eta}^k\cdot \bd{a}_3(\bd{\eta}^k)=|\partial_1\bd{\eta}^k\wedge \partial_2\bd{\eta}^k|
\]
we infer that the sequence $(\partial_1\bd{\eta}^k\wedge \partial_2\bd{\eta}^k)$ is bounded in $L^{q}(\omega;\R^3)$.
As the sequences 
$(1/R_1(\bd{\eta}^k))$ and $(1/R_2(\bd{\eta}^k))$
are bounded in $L^{\infty}(\omega)$, it follows that the sequences
$(H(\bd{\eta}^k)\partial_1\bd{\eta}^k\wedge \partial_2\bd{\eta}^k)$ and $(K(\bd{\eta}^k)\partial_1\bd{\eta}^k\wedge \partial_2\bd{\eta}^k)$ are bounded in $L^q(\omega;\R^3)$ on the one hand and on the other hand that the sequences 
$H(\bd{\eta}^k)\sqrt{a(\bd{\eta}^k)})$ and 
$(K(\bd{\eta}^k)\sqrt{a(\bd{\eta}^k)})$ are bounded in $L^q(\omega)$.

Hence, 
there exists a subsequence $(\bd{\eta}^\ell,\bd{a}_3(\bd{\eta}^\ell))$ that converges weakly to an element $( \bd{\eta},\bd{\kappa})$
in $W^{1,p}(\omega;\R^3)$. There exist also six other subsequences
\[
\left(\partial_1\bd{\eta}^\ell\wedge \partial_2\bd{\eta}^\ell\right),\quad
\left(H(\bd{\eta}^\ell)\partial_1\bd{\eta}^\ell\wedge \partial_2\bd{\eta}^\ell\right),\quad \left(K(\bd{\eta}^\ell)\partial_1\bd{\eta}^\ell\wedge \partial_2\bd{\eta}^\ell\right)
\] 
which converge weakly to $\bd{\xi}_1$, $\bd{\xi}_2$,
$\bd{\xi}_3$ in $L^q(\omega;\R^3)$ respectively and
\[
(\sqrt{a(\bd{\eta}^\ell)}),\quad (H(\bd{\eta}^\ell)\sqrt{a(\bd{\eta}^\ell)}),\quad (K(\bd{\eta}^\ell)\sqrt{a(\bd{\eta}^\ell)})
\]
which converge weakly to $\alpha_1$, $\alpha_2$, $\alpha_3$ in $L^q(\omega)$ respectively. Then, by (iii), we infer that for all $\alpha \in \{1,2\}$,
$|\varepsilon/R_{\alpha}(\bd{\eta})|\leqslant 1$ a.e. in $\omega$,
$\bd{a}_3(\bd{\eta})\in W^{1,p}(\omega;\R^3)$ and $\sqrt{a(\bd{\eta})}
\text{ in }L^{q}(\omega)$.
In order to prove that $\bd{\eta} \in \bd{\V}^{\varepsilon}$, it remains to show that
$\bd{\eta}|_{\gamma_0}=\bd{\varphi}$,
$\bd{a}_3(\bd{\eta})|_{\gamma_0}=\bd{a}_3$,
$\partial_1\bd{\eta}\wedge \partial_2\bd{\eta}\neq 0$ a.e. in $\omega$ and for all $\alpha \in \{1,2\}$, $|\varepsilon/R_{\alpha}(\bd{\eta)}|\neq 1$
a.e. in $\omega$. Since the trace operator from $W^{1,p}(\omega)$ into $L^p(\gamma_0)$ is continuous with respect to the strong topologies of both spaces, it remains so with respect to the weak topologies of both spaces. Hence, we infer from the weak convergence
$
\bd{\eta}^\ell \rightharpoonup \bd{\eta}$ and $
\bd{a}_3(\bd{\eta}^\ell) \rightharpoonup \bd{a}_3(\bd{\eta})$ in $W^{1,p}(\omega;\R^3)$
that
$\bd{\eta}^\ell|_{\gamma_0}\rightarrow  \bd{\eta}|_{\gamma_0}$ and $
\bd{a}_3(\bd{\eta}^\ell)|_{\gamma_0}\rightarrow
 \bd{a}_3(\bd{\eta})|_{\gamma_0}$ in $L^p(\gamma_0;\R^3)$
and thus 
$\bd{\eta}|_{\gamma_0}=\bd{\varphi}$ and $
\bd{a}_3(\bd{\eta})|_{\gamma_0}=\bd{a}_3$
since $\bd{\eta}^\ell|_{\gamma_0}=\bd{\varphi}$ and $\bd{a}_3(\bd{\eta}^\ell)|_{\gamma_0}=\bd{a}_3$ for all $\ell$.

In order to prove that
$\partial_1\bd{\eta}\wedge \partial_2\bd{\eta}\neq 0$ a.e. in $\omega$ and for all $\alpha \in \{1,2\}$,    $|\varepsilon/R_{\alpha}(\bd{\eta)}|\neq  1$ a.e. in
$\omega$,
it suffices to show that for all $d\in \{-1,1\}$,
\[
(1-2d\varepsilon H(\bd{\eta}) +\varepsilon^2 K(\bd{\eta}))\sqrt{a(\bd{\eta})}\neq 0 \text{ a.e. in } \omega.
\]
Assume that $(1-2\varepsilon H(\bd{\eta}) +\varepsilon^2 K(\bd{\eta}))\sqrt{a(\bd{\eta})}=0$ on a subset $A$ of $\omega$ with $\dx $-meas $A>0$. Since
$(1-2\varepsilon H(\bd{\eta}^\ell) +\varepsilon^2 K(\bd{\eta}^\ell))\sqrt{a(\bd{\eta}^\ell)}>0$ a.e. on $A$ and
\[
(1-2\varepsilon H(\bd{\eta}^\ell) +\varepsilon^2 K(\bd{\eta}^\ell))\sqrt{a(\bd{\eta}^\ell)}\rightharpoonup
(1-2\varepsilon H(\bd{\eta}) +\varepsilon^2 K(\bd{\eta}))\sqrt{a(\bd{\eta})}
\]
in $L^q(\omega)$, then
\[
\int_A (1-2\varepsilon H(\bd{\eta}^\ell) +\varepsilon^2 K(\bd{\eta}^\ell))\sqrt{a(\bd{\eta}^\ell)} \dx\rightarrow 
\int_A (1-2\varepsilon H(\bd{\eta}) +\varepsilon^2 K(\bd{\eta}))\sqrt{a(\bd{\eta})}\, \dx 
\]
by the definition of weak convergence (the characteristic function of the set $A$ belongs to the dual space of $L^{q}(\omega)$), hence
\[
(1-2\varepsilon H(\bd{\eta}^\ell) +\varepsilon^2 K(\bd{\eta}^\ell))\sqrt{a(\bd{\eta}^\ell)}\rightarrow 0\text{ in }L^1(A).
\] 
Therefore there exists a subsequence $(\bd{\eta}^m)$ of
$(\bd{\eta}^\ell)$ such that
\[
(1-2\varepsilon H(\bd{\eta}^m (x)) +\varepsilon^2 K(\bd{\eta}^m (x)))\sqrt{a(\bd{\eta}^m (x))}\rightarrow 0 \text{ for almost all }x\in A.
\]
Consider next the sequence of measurable functions $(f^m)$ defined by
\[
f^m \colon x\in A \to f^m(x):=W(x,\bd{\eta}^m).
\]
Since $f^m \geqslant C_2$ for all $m$, can apply Fatou's lemma, 
which shows that
\[
\int_A \liminf_{m\rightarrow \infty} f^m(x)\, \dx \leqslant \liminf_{m\rightarrow \infty} \int_A f^m(x)\, \dx
\]
on the one hand. On the other hand, the behavior of the function
$W$ as 
\[
(1-2\varepsilon H(\bd{\eta} (x)) +\varepsilon^2 K(\bd{\eta} (x)))\sqrt{a(\bd{\eta} (x))}\rightarrow 0^+
\]
(assumption (d)) implies that
$
\liminf_{m\rightarrow \infty} f^m(x)=\lim_{m\rightarrow \infty}
W(x,\bd{\eta}^m)=+\infty$ for almost all $x\in A$
and thus
\[
\lim_{m\rightarrow \infty}\int_A f^m(x)\, \dx =\lim_{m\rightarrow \infty}\int_A
W(x,\bd{\eta}^m)\, \dx =+\infty.
\]
But this last relation contradicts the relation 
\[
\lim_{m\rightarrow \infty}I(\bd{\eta}^m)=\inf_{\bd{\psi}\in {\bd{\V}^{\varepsilon}}}I(\bd{\psi})<+\infty
\]
and the inequalities
\begin{align*}
I(\bd{\eta}^m)\geqslant \int_A W(x,\bd{\eta}^m)\, \dx
&+C_2 \text{ area }(\omega - A)\\
&-C_3(\|\bd{\eta}^m\|_{1,p}+
\|\bd{a}_3(\bd{\eta}^m)\|_{1,p})
\end{align*}
(a weakly convergent sequence is bounded). Hence 
\[
(1-2\varepsilon H(\bd{\eta}) +\varepsilon^2 K(\bd{\eta}))\sqrt{a(\bd{\eta})}\neq 0 \text{ a.e. in }\omega,  
\]
thus
$
\partial_1\bd{\eta}\wedge \partial_2\bd{\eta}\neq 0$  a.e. in $\omega$ and for all $\alpha \in \{1,2\}$,   $\varepsilon/R_{\alpha}(\bd{\eta)}\neq  1$
 a.e. in $\omega$.
We proceed in the same manner to prove that
$(1+2\varepsilon H(\bd{\eta}) +\varepsilon^2 K(\bd{\eta}))\sqrt{a(\bd{\eta})}\neq 0$ a.e. in $\omega$,
thus we infer in addition that for all $\alpha \in \{1,2\}$,
$\varepsilon/R_{\alpha}(\bd{\eta)}\neq  -1$
a.e. in $\omega$.
To sum up, we have proved that $\bd{\eta}\in \bd{\V}^{\varepsilon}$.

(v) \emph{Finally, we show that}
\[
\int_{\omega}W(x,\bd{\eta}) \, \dx\leqslant \liminf_{\ell \to \infty}
\int_{\omega}W(x,\bd{\eta}^{\ell})\, \dx.
\]
By the definition of the limit inferior, we must show that, given any subsequence $(\bd{\eta}^m)$ of $(\bd{\eta}^\ell)$ such that the sequence $(\int_{\omega}W(x,\bd{\eta}^{m})\, \dx)$ converges, then
\[
\int_{\omega}W(x,\bd{\eta}) \, \dx\leqslant \lim_{m \to \infty}
\int_{\omega}W(x,\bd{\eta}^{m})\, \dx.
\]
So, let us consider such a subsequence. Using the results of parts (iii), (iv) and the Banach-Saks-Mazur theorem, we infer that for each $m$, there exist integers $j(m)\geqslant m$ and numbers $\mu_t^m$, $m\leqslant t \leqslant j(m)$, such that
\begin{align*}
& \mu_t^m \geqslant 0,\quad \sum_{t=m}^{j(m)}\mu_t^m=1,\\
& \bd{D}^m:=\sum_{t=m}^{j(m)} \mu_t^m \left(\nabla \bd{\eta}^t,\nabla \bd{a}_3(\bd{\eta}^t),\left(1,\varepsilon H(\bd{\eta}^t),\varepsilon^2 K(\bd{\eta}^t)\right)\sqrt{a(\bd{\eta}^t)}\right)\\
&\qquad \qquad \qquad \qquad \underset{m\to \infty}{\longrightarrow}
\left(\nabla \bd{\eta},\nabla \bd{a}_3(\bd{\eta}),\left(1,\varepsilon H(\bd{\eta}),\varepsilon^2 K(\bd{\eta})\right)\sqrt{a(\bd{\eta})}\right)
\end{align*}
in $(L^p(\omega;\mathbb{M}^{3\times 2}))^2\times  (L^q(\omega))^3$. Hence there exists a subsequence $(\bd{D}^n)$ of $(\bd{D}^m)$ such that, for almost all $x \in \omega$,
\begin{align*}
& \sum_{t=n}^{j(n)}\mu_t^n \left(\nabla \bd{\eta}^t(x),\nabla \bd{a}_3(\bd{\eta}^t(x)),\left(1,\varepsilon H(\bd{\eta}^t(x)),\varepsilon^2 K(\bd{\eta}^t(x))\right)\sqrt{a(\bd{\eta}^t(x)})\right)\\
& \qquad \qquad \underset{n\to \infty}{\longrightarrow}
\left(\nabla \bd{\eta}(x),\nabla \bd{a}_3(\bd{\eta}(x)),\left(1,\varepsilon H(\bd{\eta}(x)),\varepsilon^2 K(\bd{\eta}(x))\right)\sqrt{a(\bd{\eta}(x)})\right).
\end{align*}
Since the function $\mathbb{W}(x,\cdot)$ is continuous on the set
\[
\bd{\M}:=\{( \bd{A},
\bd{B},a,b,c) \in (\mathbb{M}^{3\times 2})^2 \times \R^3, \,
a-|b|>0 \text{ and } a-2|b| +c >0 \}
\]
for almost all $x\in \omega$ and since $\bd{\eta}\in \bd{\V}^{\varepsilon}$ by part (iv), it follows that for almost all $x \in \omega$
\[
\left(\nabla \bd{\eta}(x),\nabla \bd{a}_3(\bd{\eta}(x)),\left(1,\varepsilon H(\bd{\eta}(x)),\varepsilon^2 K(\bd{\eta}(x))\right)\sqrt{a(\bd{\eta}(x)})\right)\in \bd{\M}
\]
and
\begin{align*}
W(x,\bd{\eta})&=\mathbb{W}\left(x,\nabla \bd{\eta}(x),\nabla \bd{a}_3(\bd{\eta}(x)),\left(1,\varepsilon H(\bd{\eta}(x)),\varepsilon^2 K(\bd{\eta}(x))\right)\sqrt{a(\bd{\eta}(x)})\right)\\
&=\lim_{n \to \infty}\mathbb{W}\left(x,\sum_{t=n}^{j(n)}\mu_t^n \left(\nabla \bd{\eta}^t(x),\nabla \bd{a}_3(\bd{\eta}^t(x)),\bd{v}(\bd{\eta}^t(x))\right)\right)
\end{align*}
where $\bd{v}(\bd{\eta}^t(x)):=(1,\varepsilon H(\bd{\eta}^t(x)),\varepsilon^2 K(\bd{\eta}^t(x)))\sqrt{a(\bd{\eta}^t(x)})$.
Using this relation, Fatou's lemma, and the assumed convexity of the function $\mathbb{W}(x,\cdot)$ for almost all $x \in \omega$, we next obtain, on the one hand,
\begin{align*}
\int_{\omega}W(x,\bd{\eta})\, \dx &\leqslant \liminf_{n \to \infty}\int_{\omega} \mathbb{W}\left(x,\sum_{t=n}^{j(n)}\mu_t^n \left(\nabla \bd{\eta}^t(x),\nabla \bd{a}_3(\bd{\eta}^t(x)),\bd{v}(\bd{\eta}^t(x))\right)\right)\, \dx\\
&\leqslant \liminf_{n \to \infty}\sum_{t=n}^{j(n)}\mu_t^n 
\int_{\omega} W(x,\bd{\eta}^t)\, \dx=\lim_{n \to \infty}
\int_{\omega}W(x,\bd{\eta}^n)\, \dx\\
&=\lim_{m \to \infty}\int_{\omega}W(x,\bd{\eta}^m)\, \dx.
\end{align*}
Since, on the other hand, $L(\bd{\eta},\bd{a}_3(\bd{\eta}))=\lim_{\ell \to \infty}L(\bd{\eta}^\ell,\bd{a}_3(\bd{\eta}^\ell))$ by definition of weak convergence, we have thus proved that
$
I(\bd{\eta})\leqslant \liminf_{\ell \to \infty}I(\bd{\eta}^\ell)$.

(vi) \emph{The function $\bd{\eta}$ is thus a solution of the minimization problem}, since $\bd{\eta}\in \bd{\V}^{\varepsilon}$ by parts (iii) and (iv), and since
\[
I(\bd{\eta})\leqslant \liminf_{\ell \to \infty}I(\bd{\eta}^\ell)=\inf_{\bd{\psi}\in \bd{\V}^{\varepsilon}}I(\bd{\psi})\quad \text{implies}\quad I(\bd{\eta})=\inf_{\bd{\psi}\in \bd{\V}^{\varepsilon}}I(\bd{\psi}).
\]
\end{proof}

As an example of polyconvex stored energy function , there is the Helfrich energy (see Helfrich \cite{Helfrich73}) given by
\[
W(\bd{\psi})=\left( \frac{k_c}{2} (2H(\bd{\psi})+c_0)^2+\overline{k}K(\bd{\psi})+\lambda \right) \sqrt{a(\bd{\psi})}
\]
used for modeling biomembranes, where $k_c>0$ and $\overline{k}\in \R$ denote bending rigidities, $c_0 \in \R$ stands for the spontaneous curvature and $\lambda\in \R$ is the surface tension.

\section{Stored energy functions for $\boldsymbol{G^1}$ shells}
\label{Example}

Let us first define a $G^1$ shell with thickness $2 \varepsilon >0$. This regularity 
has been first introduced in Anicic \cite{Anicic01, Anicic03}. The term $G^1$ stands for \emph{First-Order Geometric Continuity}. 

The midsurface of the reference configuration of a shell is denoted by $S:=\bd{\varphi}(\omega)$ where
\begin{equation}
\bd{\varphi} \in W^{1,\infty}(\omega;\R^3).
\label{hyp_reg1}
\end{equation}
The two vectors $\bd{a}_\alpha := \partial_{\alpha} \bd{\varphi}\in L^{\infty}(\omega;\R^3)$ span the tangent plane to the surface $S$. We suppose that $\bd{\varphi}$ satisfies the additional assumption
\begin{equation}
\underset{\omega}{\textnormal{ess inf }} |\bd{a}_1\wedge \bd{a}_2| >0 \textnormal{ and } 
\bd{a}_3:=\frac{\bd{a}_1\wedge \bd{a}_2}{|\bd{a}_1\wedge \bd{a}_2|} \in W^{1,\infty}(\omega;\R^3),
\label{hyp_reg2}
\end{equation}
where $\bd{a}_3$ is the unit normal vector to the surface $S$. 

The covariant components $a_{\alpha \beta}\in L^{\infty}(\omega)$, $b_{\alpha \beta}\in L^{\infty}(\omega)$ and  $c_{\alpha \beta}\in L^{\infty}(\omega)$ of the first, second and third fundamental forms of $S$ are respectively defined by
$a_{\alpha \beta}:=\bd{a}_{\alpha} \cdot \bd{a}_{\beta}$, 
$b_{\alpha \beta}:=-\bd{a}_{\alpha} \cdot \partial_{\beta}\bd{a}_{3}=-\bd{a}_{\beta} \cdot \partial_{\alpha}\bd{a}_{3}$ and 
$c_{\alpha \beta}:=\partial_{\alpha}\bd{a}_{3} \cdot \partial_{\beta}\bd{a}_{3}$.
The area
element along $S$ is $\sqrt{a}\,dx$, whith
\[
a:=\det(a_{\alpha \beta})=|\bd{a}_1 \wedge \bd{a}_2|^2.
\]
Since 
$a$ is uniformly bounded from below on $\omega$, 
the inverse of the matrix $(a_{\alpha \beta})$, which we denote $(a^{\alpha \beta})$, belongs to $L^{\infty}(\omega)$. The contravariant basis $\bd{a}^{\alpha}\in L^{\infty}(\omega;\R^3)$ is then defined by letting 
$
\bd{a}^{\alpha}=a^{\alpha \beta}\bd{a}_{\beta}
$
and then satisfy
$
\bd{a}^{\alpha}\cdot \bd{a}_{\beta}=\delta^{\alpha}_{\beta}$,
where $\delta^{\alpha}_{\beta}$ is the Kronecker symbol. 
The mixed components $b_{\alpha}^{ \beta}\in L^{\infty}(\omega)$ of
the second fundamental form are defined by
$
b_{\alpha}^{ \beta}:=b_{\alpha\rho}a^{\rho \beta}$.
The mean curvature $H\in L^{\infty}(\omega)$ and the Gaussian curvature $K \in L^{\infty}(\omega)$ are respectively defined by
\begin{align*}
H:=\frac{1}{2}(b_1^1+b_2^2)=\frac{1}{2}\left(\frac{1}{R_1}+\frac{1}{R_2}\right) \quad \text{and}\quad K:=b_1^1b_2^2-b_1^2b_2^1=\frac{1}{R_1R_2},
\end{align*}
where the invariants $1/R_{\alpha}$ are the principal curvatures of $S$.

The reference configuration of a shell with thickness $2\varepsilon >0$ is the set
\[
\{ \bd{\Phi}(x,z):=\bd{\varphi}(x)+z\bd{a}_3(x); \enspace (x,z)\in \Omega:=\omega \times (-\varepsilon,\varepsilon)\}.
\]
The tangent vectors are respectively defined by
\[
\bd{g}_{\alpha}(x,z):=\partial_{\alpha}\bd{\Phi}(x,z)=\bd{a}_{\alpha}(x)+z\partial_{\alpha}\bd{a}_3(x).
\] 
Then
\[
\det \nabla \bd{\Phi}(x,z)=\left(1-\frac{z}{R_1(x)} \right)\left(1-\frac{z}{R_2(x)} \right)\sqrt{a(x)}.
\]
In addition to the hypotheses (\ref{hyp_reg1})-(\ref{hyp_reg2}), 
we also impose that $\bd{\varphi}$ and $\varepsilon$ satisfy the following assumption:
\begin{equation}
\underset{\Omega}{\textnormal{ess inf }}\det \nabla \bd{\Phi}>0.
\label{hyp_reg3}
\end{equation}
The contravariant basis $\bd{g}^{\alpha}\in L^{\infty}(\Omega;\R^3)$ is defined by
$\bd{g}^{\alpha}\cdot \bd{g}_{\beta}=\delta^{\alpha}_{\beta}$.

To sum up, equivalently to the hypotheses (\ref{hyp_reg1})-(\ref{hyp_reg2})-(\ref{hyp_reg3}), we define a $G^1$ shell with thickness $2\varepsilon>0$ a shell whose midsurface $S:=\bd{\varphi}(\omega)$ satisfies $\bd{\varphi}\in G^1$ where
\begin{align*}
G^1:= \{  \bd{\varphi}\in W^{1,\infty}(\omega;\R^3);\enspace &
(\bd{a}_1\wedge \bd{a}_2)/|\bd{a}_1\wedge \bd{a}_2| \in W^{1,\infty}(\omega;\R^3),\\
& \underset{\omega}{\textnormal{ess inf }}|\bd{a}_1\wedge \bd{a}_2|>0, \enspace \underset{\alpha \in \{1,2\}}{\textnormal{max}}|\varepsilon /R_{\alpha}|_{\infty,\omega}<1 \}.
\end{align*}
This regularity allows us to take into account curvature discontinuities as well as  tangent plane continuity, even if the tangent vectors are not continuous. Hence, if we consider a surface which is defined via smooth patches, we are only led to match the unit normal vectors on the interfaces and not the tangent vectors. This makes for great versatility in practice. Moreover, this regularity does not involve any Christoffel symbols.

Let us now define a class of polyconvex stored energy functions for $G^1$ shells which satisfies a coerciveness inequality.

\begin{theorem}
Let 
\[
\bd{\mathrm{N}}:=\{(a,b,c) \in \R^3, \,
 a-|b|>0 \text{ and } a-2|b| +c >0 \}
\]
and $\Gamma :\omega \times \bd{\mathrm{N}} \rightarrow \R$ be a function such that $\Gamma(x,\cdot) : \bd{\mathrm{N}} \rightarrow \R$ is convex for almost all $x\in \omega$. Let 
$W:\omega\times \bd{\V}^{\varepsilon}\rightarrow \R$ be a stored energy function defined by
\begin{align*}
W(x,\bd{\psi})=\sum_{i=1}^R &\left\{a_i \, \mathrm{tr}\left(\bd{G}(x,\bd{\psi},u_i,v_i)^{\gamma_i/2}\right)+b_i \, \mathrm{tr}\left(\bd{G}(x,\bd{\psi},-u_i,w_i)^{\gamma_i/2}\right)\right\} \\
&\qquad \quad + \Gamma(x,\sqrt{a(\bd{\psi})},\varepsilon H(\bd{\psi})\sqrt{a(\bd{\psi})},\varepsilon^2 K(\bd{\psi})\sqrt{a(\bd{\psi})})
\end{align*}
where 
\[
\bd{G}(x,\bd{\psi},u,v):=\{a_{\alpha \beta}(\bd{\psi})-2u b_{\alpha \beta}(\bd{\psi})+u^2 c_{\alpha \beta}(\bd{\psi})\}\,\bd{g}^{\alpha} (x,v) \otimes \bd{g}^{\beta}(x,v)
\]
and $a_i>0$, $b_i>0$, $\gamma_i\geqslant 2$, $u_i\in \R$, $(v_i,w_i)\in [-\varepsilon,\varepsilon]^2$,  $1\leqslant i \leqslant R$.

Then the function $W$ is polyconvex and satisfies a coerciveness inequality of the form
\begin{align*}
W(x,\bd{\psi})\geqslant C\{ |\nabla \bd{\psi}|^{\gamma_{i_0}}&+|u_{i_0}|^{\gamma_{i_0}}| \nabla \bd{a}_{3}(\bd{\psi})|^{\gamma_{i_0}}\}\\
&+\Gamma(x,\sqrt{a(\bd{\psi})},\varepsilon H(\bd{\psi})\sqrt{a(\bd{\psi})},\varepsilon^2 K(\bd{\psi})\sqrt{a(\bd{\psi})})
\end{align*}
with a constant $C>0$ and ${\gamma_{i_0}}=\max_i\gamma_i$.
\end{theorem}

\begin{proof}
Let
$\bd{D}(x,v):=\bd{e}_{\alpha}\otimes \bd{g}^{\alpha}(x,v)$ where $(\bd{e}_1,\bd{e}_2)$ denotes the canonical basis of $\mathbb{R}^2$ and $|v|\leqslant \varepsilon$.
Then
\[
\bd{G}(x,\bd{\psi},u,v)=\bd{D}(x,v)^{\mathrm{T}}(\nabla \bd{\psi}+u
\nabla \bd{a}_3(\bd{\psi}))^{\mathrm{T}}(\nabla \bd{\psi}+u
\nabla \bd{a}_3(\bd{\psi}))\bd{D}(x,v).
\]
As a composition of a linear function
and a convex function, the function $\mathbb{F}(x,u,v,\cdot) : (\mathbb{M}^{3\times 2})^2 \rightarrow \mathbb{R}$ defined by
\[
\mathbb{F}(x,u,v,\bd{A},\bd{B}):=\mathrm{tr}\left(\left\{\bd{D}(x,v)^{\mathrm{T}}(\bd{A}+u\bd{B})^{\mathrm{T}}(\bd{A}+u\bd{B})\bd{D}(x,v)\right\}^{\gamma/2}\right)
\]
is convex for all $\gamma \geqslant 2$. By noting that
\[
\mathrm{tr}\left( \bd{G}(x,\bd{\psi},u,v)^{\gamma/2}\right)=\mathbb{F}(x,u,v,\nabla \bd{\psi},\nabla \bd{a}_3(\bd{\psi})),
\]
we infer that $W$ is polyconvex.

It remains to prove the coerciveness inequality. By the equivalence of norms on a finite-dimensional space, it follows that  for each $\gamma \geqslant 1$ there exists a constant $C_{\gamma}>0$ such that
\[
\mathrm{tr}\left( \bd{G}(x,\bd{\psi},u,v)^{\gamma/2}\right)\geqslant
 C_{\gamma}\left(\mathrm{tr}\,\bd{G}(x,\bd{\psi},u,v)\right)^{\gamma/2}.
\]
Since $\mathrm{tr}\,\bd{G}(x,\bd{\psi},u,v)=
|(\nabla \bd{\psi}+u
\nabla \bd{a}_3(\bd{\psi}))\bd{D}(x,v)|^2$ and
\[
\nabla \bd{\psi}+u
\nabla \bd{a}_3(\bd{\psi})=(\nabla \bd{\psi}+u
\nabla \bd{a}_3(\bd{\psi}))\bd{D}(x,v)(\nabla \bd{\varphi}+v \nabla \bd{a}_3),
\]
we infer that
\[
|\nabla \bd{\psi}+u
\nabla \bd{a}_3(\bd{\psi})|\leqslant |(\nabla \bd{\psi}+u
\nabla \bd{a}_3(\bd{\psi}))\bd{D}(x,v)|
(\|\nabla \bd{\varphi}\|_{\infty,\omega}+\varepsilon\| \nabla \bd{a}_3\|_{\infty,\omega})
\]
and that there exists a constant $C>0$ such that for all $u\in \R$ and all $(v,w)\in [-\varepsilon,\varepsilon]^2$,
\begin{align*}
&\mathrm{tr}\left( \bd{G}(x,\bd{\psi},u,v)^{\gamma/2}\right)+ \mathrm{tr}\left(  \bd{G}(x,\bd{\psi},-u,w)^{\gamma/2}\right)\\
& \qquad \qquad \qquad \qquad \quad \geqslant C\{|\nabla \bd{\psi}+u
\nabla \bd{a}_3(\bd{\psi})|^{2}+|\nabla \bd{\psi}-u
\nabla \bd{a}_3(\bd{\psi})|^{2}\}^{\gamma /2}.
\end{align*}
The coerciveness inequality follows by noting that
\[
|\nabla \bd{\psi}+u
\nabla \bd{a}_3(\bd{\psi})|^{2}+|\nabla \bd{\psi}-u
\nabla \bd{a}_3(\bd{\psi})|^{2}=2|\nabla \bd{\psi}|^{2}+2u^2|\nabla \bd{a}_3(\bd{\psi})|^{2}.
\]
\end{proof}

\end{document}